\ifdefined\journalstyle
\documentclass[12pt]{iopart}
\expandafter\let\csname equation*\endcsname\relax
\expandafter\let\csname endequation*\endcsname\relax
\else
  \documentclass[11pt,a4paper]{article}
\fi
\usepackage{amsmath,amsthm,amssymb,url,mathrsfs,graphicx,color,enumitem,scrextend,microtype}

\newcommand{\setu}{{\mathrm{\mathfrak{u}}}}

\usepackage[font=scriptsize]{caption}
\newcommand{\bsx}{{\boldsymbol{x}}}
\newcommand{\bsy}{{\boldsymbol{y}}}

\newcommand{\bst}{{\boldsymbol{t}}}

\newcommand{\bsgamma}{{\boldsymbol{\gamma}}}

\newcommand{\bseta}{{\boldsymbol{\eta}}}
\newcommand{\bsnu}{{\boldsymbol{\nu}}}

\newcommand{\bsb}{{\boldsymbol{b}}}

\newcommand{\mask}[1]{{}}

\newtheorem{theorem}{Theorem}
\newtheorem{lemma}[theorem]{Lemma}
\newtheorem{corollary}[theorem]{Corollary}
\newtheorem{assumption}[theorem]{Assumption}
\numberwithin{theorem}{section}
\usepackage{pgfplots}
\usepackage{float,subfig}
\pgfplotsset{compat=1.18}

\ifdefined\journalstyle
\else
\title{Uncertainty quantification for electrical impedance tomography using quasi-Monte Carlo methods}
\author{Laura Bazahica\footnotemark[2]\and Vesa Kaarnioja\footnotemark[3]\and Lassi Roininen\footnotemark[2]}
\fi

\begin{document}

\ifdefined\journalstyle
\title[Uncertainty quantification for electrical impedance tomography using QMC]{Uncertainty quantification for electrical impedance tomography using quasi-Monte Carlo methods}
\author{Laura Bazahica$^1$, Vesa Kaarnioja$^2$ and Lassi Roininen$^1$}
\address{$^1$ School of Engineering Sciences, LUT University, P.O.~Box 20, 53851 Lappeenranta, Finland\\
$^2$ Department of Mathematics and Computer Science, Free University of Berlin, Arnimallee 6, 14195 Berlin, Germany}
\ead{laura.bazahica@lut.fi, vesa.kaarnioja@fu-berlin.de \textnormal{and} lassi.roininen@lut.fi}
\vspace{10pt}
\begin{indented}
\item[]\today
\end{indented}
\else
\maketitle
\renewcommand{\thefootnote}{\fnsymbol{footnote}}
\footnotetext[2]{School of Engineering Sciences, LUT University, P.O.~Box 20, 53851 Lappeenranta, Finland. Email: {\tt \{laura.bazahica,lassi.roininen\}@lut.fi}}
\footnotetext[3]{Department of Mathematics and Computer Science, Free University of Berlin, Arnimallee 6, 14195 Berlin, Germany. Email: {\tt vesa.kaarnioja@fu-berlin.de}}
\fi

\begin{abstract}
The theoretical development of quasi-Monte Carlo (QMC) methods for uncertainty quantification of partial differential equations (PDEs) is typically centered around simplified model problems such as elliptic PDEs subject to homogeneous zero Dirichlet boundary conditions. In this paper, we present a theoretical treatment of the application of randomly shifted rank-1 lattice rules to electrical impedance tomography (EIT). EIT is an imaging modality, where the goal is to reconstruct the interior conductivity of an object based on electrode measurements of current and voltage taken at the boundary of the object. This is an inverse problem, which we tackle using the Bayesian statistical inversion paradigm. As the reconstruction, we consider QMC integration to approximate the unknown conductivity given current and voltage measurements. We prove under moderate assumptions placed on the parameterization of the unknown conductivity that the QMC approximation of the reconstructed estimate has a dimension-independent, faster-than-Monte Carlo cubature convergence rate. Finally, we present numerical results for examples computed using simulated measurement data.
\end{abstract}
%
\vspace{2pc}
\noindent{\it Keywords}: Electrical impedance tomography, Bayesian inversion, complete electrode model, inaccurate measurement model, uncertainty quantification, quasi-Monte Carlo method
%
%
%
%

\section{Introduction}
Quasi-Monte Carlo (QMC) methods are a class of high-dimensional cubature rules with equal cubature weights. QMC methods have become a popular tool in the numerical treatment of uncertainties in partial differential equation (PDE) models with random or uncertain inputs. Studied topics include elliptic eigenvalue problems~\cite{chernovle2, GGKSS2019}, optimal control~\cite{guth21,guth22}, various diffusion problems~\cite{chernovle1,kuonuyenssurvey,log5}, parametric operator equations~\cite{dicklegiaschwab,spodpaper14} as well as elliptic PDEs with uniformly random or lognormal coefficients~\cite{ghs18,log,log2,log3,log4,kks20,kss12,kssmultilevel}.
A common and sought-after advantage of these applications is the faster-than-Monte Carlo convergence rate of QMC methods and---under some moderate conditions---this convergence can be shown to be independent of the dimensionality of the associated integration problems. QMC methods are particularly well-suited to large-scale uncertainty quantification problems since it is typically easy to parallelize the computation over QMC point sets. In the inverse setting, QMC integration can be used for the estimation of quantities of interest expressed as high-dimensional integrals. The uncertainty corresponding to these estimators can in turn be quantified using, e.g., credible sets.

The overwhelming majority of QMC analyses for PDEs have been carried out within the context of the forward uncertainty quantification problem of approximating the expected value $$\mathbb E[G(u)]=\int_{\Omega}G(u(\cdot,\omega))\,\mathbb P({\rm d}\omega),$$
where $(\Omega,\mathcal A,\mathbb P)$ is a probability space and $G$ is a bounded linear functional acting on the solution $u$ of an elliptic PDE with a random diffusion coefficient $a$ such as
\begin{align}
\begin{cases}
-\nabla\cdot (a(\bsx,\omega)\nabla u(\bsx,\omega))=f(\bsx),&\bsx\in D,~\omega\in \Omega,\\
u(\bsx,\omega)=0,&\bsx\in\partial D,~\omega\in\Omega,
\end{cases}\label{eq:modelproblempde}
\end{align}
where $f\!:D\to\mathbb R$ is a fixed source term. For mathematical analysis, the boundary conditions are typically taken to be extremely simple such as the homogeneous zero Dirichlet boundary condition in~\eqref{eq:modelproblempde}, which is an unrealistic modeling assumption for most practical applications.

In addition, there have been some studies on Bayesian inversion governed by parametric PDEs \cite{dglgs19,gantnerphd,gantnerpeters,hks21,scheichlbayes,ssw20}. However, these studies generally restrict their attention to simplified model problems such as~\eqref{eq:modelproblempde}.

Meanwhile, electrical impedance tomography (EIT) is an imaging modality, which uses measurements of voltage and current taken over an array of electrodes placed on the boundary of an object, with the goal of reconstructing the unknown conductivity inside the object; see for example~\cite{adler2021electrical,borcea}. The most accurate way to model the measurements of EIT is to employ the \emph{complete electrode model} (CEM), which accounts for the electrode shapes and contact resistances caused by resistive layers at electrode-object interfaces~\cite{cheng,hyvonen,somersalo}. The mathematical model of CEM is governed by an elliptic PDE---albeit one with more intricate boundary conditions than the problem~\eqref{eq:modelproblempde}---which suggests that one can apply QMC theory for both forward and inverse uncertainty quantification for this problem. However, this kind of analysis does not appear to have been considered in the literature. It is the goal of this paper to rectify this situation. Specifically, we focus on the theoretical development of QMC rules for the computation of posterior expectations.

We remark that EIT is an exponentially ill-posed problem only allowing logarithmic stability estimates~\cite{alessandrini,mandache}, meaning that the reconstruction is extremely sensitive to small changes in the measurements. In this work we focus carrying out QMC analysis for integration over the posterior. Additionally, the reconstruction quality might be further improved by combining QMC with another technique such as importance sampling~\cite{he2024quasi,bartuska} or Laplace approximation~\cite{ssw20}. These techniques also help remedy localization issues with the posterior distribution as the amount of data increases or the noise level decreases.

This paper is organized as follows. We briefly review basic multi-index notation in Section~\ref{sec:notations}. We define the problem setting---the so-called \emph{parametric complete electrode model}---in Section~\ref{sec:pcem}. The parametric regularity analysis for both the forward model and inverse problem, i.e., the conditional mean estimator of the unknown conductivity, is carried out in Section~\ref{sec:parametricanalysis}. The basics of randomly shifted rank-1 lattice rules are discussed in Section~\ref{sec:qmc}, and the error analysis for the EIT problem is presented in Section~\ref{sec:error}. Numerical experiments, including credible sets calculated to quantify the uncertainty in the obtained estimators, are showcased in Section~\ref{sec:numex}. Finally, some conclusions and future prospects are given in Section~\ref{sec:conclusions}.

\subsection{Notations and preliminaries}\label{sec:notations}
We will use boldfaced symbols to denote multi-indices while the subscript notation $\nu_j$ is used to refer to the $j^{\rm th}$ component of a multi-index $\boldsymbol \nu$. The set of finitely supported multi-indices is denoted by
$$
\mathscr F:=\{\boldsymbol\nu\in\mathbb N_0^{\mathbb N}:|\bsnu|<\infty\},
$$
where the {\em order} of a multi-index $\boldsymbol\nu\in\mathscr F$ is defined as
$$
|\bsnu|:=\sum_{j\geq 1}\nu_j.
$$
For any sequence $\bsx:=(x_j)_{j\geq 1}$ of real numbers, we define
$$
\bsx^{\boldsymbol \nu}:=\prod_{j\geq 1}x_j^{\nu_j},
$$
where we use the convention $0^0:=1$.

Let $\boldsymbol m,\bsnu\in\mathscr F$ be multi-indices. We define $\boldsymbol m\leq\bsnu$ to mean $m_j\leq\nu_j$ for all $j\geq 1$. Finally, we define the shorthand notations
$$
\partial_{\bsy}^{\bsnu}:=\prod_{j\geq 1}\frac{\partial^{\nu_j}}{\partial y_j^{\nu_j}}\quad\text{and}\quad \binom{\bsnu}{\boldsymbol m}:=\prod_{j\geq 1}\binom{\nu_j}{m_j}.
$$

\section{Parametric complete electrode model}\label{sec:pcem}
 Let $D\subset\mathbb{R}^d$, $d\in\{1,2,3\}$, denote a nonempty, bounded physical domain with Lipschitz boundary and let $\Upsilon:=(-1/2,1/2)^{\mathbb{N}}$ denote a set of parameters. We state the following assumptions about the parametric conductivity field. %
 \begin{assumption}\label{assumption}
 \rm The parametric conductivity coefficient $a\!: D\times \Upsilon\to \mathbb R$ satisfies the following assumptions:
 \begin{addmargin}[1em]{0em}
 \begin{enumerate}[label=(A\arabic*)]
 \item $a(\cdot,\bsy)\in L^\infty(D)$ for all $\bsy\in\Upsilon$.\label{a1}
 \item There exist constants $C_a,\sigma\geq 1$ and a sequence $\boldsymbol \rho=(\rho_j)_{j\geq 1}\in\ell^1(\mathbb N)$ of non-negative real numbers such that\label{a2}
 $$
 \|\partial_{\bsy}^{\bsnu}a(\cdot,\bsy)\|_{L^\infty(D)}\leq C_a(|\bsnu|!)^\sigma \boldsymbol \rho^{\bsnu}\quad\text{for all}~\bsnu\in\mathscr F~\text{and}~\bsy\in \Upsilon.
 $$
 \item There exist positive constants $a_{\min}$ and $a_{\max}$ such that\label{a3}$$0<a_{\min}\leq a(\bsx,\bsy)\leq a_{\max}<\infty\quad\text{for all $\bsx\in D$ and $\bsy\in \Upsilon$.}$$
 \end{enumerate}
 \end{addmargin}
 \end{assumption}
 Assumptions~\ref{a1} and~\ref{a3} ensure that the coefficient $a$ is uniformly elliptic over the parameter domain $\Upsilon$. Meanwhile, assumption~\ref{a2} asserts that $a$ belongs to the {\em Gevrey class} with parameter $\sigma$ with respect to the variable $\bsy\in \Upsilon$. This class has recently been studied within the context of forward uncertainty quantification for elliptic PDEs~\cite{chernovle2,chernovle1,doi:10.1142/S0218202524500179} with homogeneous Dirichlet boundary conditions. In contrast to other regularity classes analyzed within the context of inverse problems based on Fr\'echet differentiability  \cite{alberti2022inverse, harrach2019uniqueness}, the Gevrey class is based on smooth functions with bounded directional derivatives in the sense of Gateaux. Infinite differentiability with respect to Banach spaces has been considered by~\cite{garde,lechleiter}. However, we require information about the decay of directional derivatives in order to track the sparsity of the integrand, which will ultimately be used in the design of tailored QMC lattice rules in Section~\ref{sec:error}. The Gevrey class covers this as well as a wide range of possible parameterizations for the input random field $a$, which enable the development of dimension-robust QMC cubatures for uncertainty quantification.

In this work, we investigate the parametric regularity for EIT, the mathematical model of which is more complex than these models while being physically motivated. While one could take the analysis farther by incorporating more sophisticated boundary conditions, we limit ourselves to the current setting to keep the challenges of tracking additional terms more reasonable.

Let $\{E_k\}_{k=1}^M$, $M\geq 2$, be open, nonempty, connected subsets of $\partial D$ such that $\overline{E}_i\cap\overline{E}_j=\varnothing$ for $i\neq j$. We define the quotient Hilbert space $\mathcal{H}:=(H^1(D)\oplus \mathbb{R}^M)/\mathbb R$, endowed with the norm
$$
\|(v,V)\|_{\mathcal{H}}^2:=\int_D|\nabla v|^2\,{\rm d}\bsx+\sum_{k=1}^M\int_{E_k}(v-V_m)^2\,{\rm d}S,\quad (v,V)\in\mathcal{H}.
$$
Note that $(u,U)=(v,V)\in\mathcal H$, if
$$
u-v={\rm constant}=U_1-V_1=\cdots=U_M-V_M.
$$

Let $a\!:D\times \Upsilon\to \mathbb R$ be a parametric conductivity field satisfying Assumption~\ref{assumption}. The forward problem is to find, for $\bsy\in \Upsilon$, the electromagnetic potential $u(\cdot,\bsy)$ and the potentials on the electrodes $U(\bsy)$ satisfying
\begin{align*}
\begin{cases}
\nabla\cdot (a(\cdot,\bsy)\nabla u(\cdot,\bsy))=0&\text{in}~D,\\
a(\cdot,\bsy)\frac{\partial u(\cdot,\bsy)}{\partial \boldsymbol n}=0&\text{on}~\partial D\setminus\overline{E},\\
u(\cdot,\bsy)+z_ma(\cdot,\bsy)\frac{\partial u(\cdot,\bsy)}{\partial \boldsymbol n}=U_m(\bsy)&\text{on}~E_m,~m\in\{1,\ldots,M\},\\
\int_{E_m}a(\bsx,\bsy)\frac{\partial u(\bsx,\bsy)}{\partial \boldsymbol n}\,{\rm d}S(\bsx)=I_m&m\in\{1,\ldots,M\},
\end{cases}
\end{align*}
where $\boldsymbol n=\boldsymbol n(\bsx)$ denotes the outward pointing unit normal vector for $\bsx\in\partial D$, the vector $I\in\mathbb R^M$ consists of the net current feed through the electrodes with $\sum_{m=1}^M I_m=0$, and the real-valued contact impedances $\{z_m\}_{m=1}^M$ are assumed to satisfy, for some positive constants $\varsigma_-$ and $\varsigma_+$, the inequalities $0<\varsigma_-\leq z_m\leq \varsigma_+<\infty$ for all $m\in\{1,\ldots,M\}$.

The variational formulation of the above is: for $\bsy\in \Upsilon$, find $(u(\cdot,\bsy),U(\bsy))\in \mathcal{H}$ such that
\begin{align}
B((u(\cdot,\bsy),U(\cdot,\bsy)),(v,V))=\sum_{m=1}^MI_mV_m\quad\text{for all}~(v,V)\in \mathcal{H},\label{eq:variational}
\end{align}
where
$$
B((w,W),(v,V))=\int_Da(\cdot,\bsy)\nabla w\cdot \nabla v\,{\rm d}\bsx+\sum_{m=1}^M\frac{1}{z_m}\int_{E_m}(w-W_m)(v-V_m)\,{\rm d}S.
$$ It is known that, under the aforementioned assumptions, a unique solution to the variational formulation exists~\cite{somersalo}. Moreover, the solution satisfies the {\em a priori} bound (cf.~\cite[Lemma~2.1]{stratos})
\begin{align*}
\|(u(\cdot,\bsy),U(\bsy))\|_{\mathcal{H}}\leq \frac{C|I|}{\min\{a_{\min},\varsigma_+^{-1}\}}\quad\text{for all}~\bsy\in \Upsilon,%
\end{align*}
where $C>0$ is a constant depending on $D$ and $E$. We call the model~\eqref{eq:variational} the {\em parametric complete electrode model} ({\em pCEM}). The pCEM was introduced in the numerical study~\cite{pcempaper}, but analytic properties such as parametric regularity were not investigated.

\section{Parametric regularity}\label{sec:parametricanalysis}
\subsection{Forward problem}
We begin by establishing the existence of the higher-order partial derivatives of the solution to~\eqref{eq:variational}. The following lemma shows that the first-order partial derivatives are well-defined. To this end, we generalize the proof strategy of~\cite[Theorem~4.2]{cohen} for nonlinear parameterizations of the input coefficient $a$. %
\begin{lemma} \label{lemma1} Let Assumption~\ref{assumption} hold and let $(u(\cdot,\bsy),U(\bsy))\in\mathcal H$ be the solution to~\eqref{eq:variational}. Then $\partial_{y_j}(u(\cdot,\bsy),U(\bsy))$ exists and belongs to $\mathcal{H}$ for all $j\geq1$ and $\bsy\in \Upsilon$.
\end{lemma}
\begin{proof} 
Let $\boldsymbol e_j$ denote the unit vector with $1$ at index $j$ and $0$ otherwise. For $h\in\mathbb R\setminus\{0\}$ define the difference quotients
\begin{align*}
w_h(\cdot,\bsy)&:=\frac{u(\cdot,\bsy+h\boldsymbol e_j)-u(\cdot,\bsy)}{h},\\
W_h(\bsy)&:=\frac{U(\bsy+h\boldsymbol e_j)-U(\bsy)}{h},\\
\alpha_h(\cdot,\bsy)&:=\frac{a(\cdot,\bsy+h\boldsymbol e_j)-a(\cdot,\bsy)}{h}.
\end{align*}
For sufficiently small $|h|<1$,
$$
a_{\min}\leq a(\bsx,\bsy+h\boldsymbol e_j)\leq a_{\max},\quad \bsx\in D,
$$
which means that $u(\cdot,\bsy+h\boldsymbol e_j)$ is well-defined as an element of $\mathcal H$. Then we can conclude that for all $(v,V)\in\mathcal H$, there holds
\begin{align*}
0&=\int_D a(\bsx,\bsy+h\boldsymbol e_j)\nabla u(\bsx,\bsy+h\boldsymbol e_j)\cdot \nabla v(\bsx)\,{\rm d}\bsx\!-\! \int_D a(\bsx,\bsy)\nabla u(\bsx,\bsy+h\boldsymbol e_j)\cdot \nabla v(\bsx)\,{\rm d}\bsx\\
&\quad + \int_D a(\bsx,\bsy)\nabla u(\bsx,\bsy+h\boldsymbol e_j)\cdot \nabla v(\bsx)\,{\rm d}\bsx - \int_D a(\bsx,\bsy)\nabla u(\bsx,\bsy)\cdot \nabla v(\bsx)\,{\rm d}\bsx\\
&\quad + \sum_{m=1}^M \frac{1}{z_m}\int_{E_m}((u(\bsx,\bsy+h\boldsymbol e_j)-u(\bsx,\bsy))-((U(\bsy+h\boldsymbol e_j)-U(\bsy)))(v(\bsx)-V)\,{\rm d}S(\bsx)\\
&=h\int_D \alpha_h(\bsx,\bsy)\nabla u(\bsx,\bsy+h\boldsymbol e_j)\cdot \nabla v(\bsx)\,{\rm d}\bsx+h\int_D a(\bsx,\bsy)\nabla w_h(\bsx,\bsy)\cdot \nabla v(\bsx)\,{\rm d}\bsx\\
&\quad + h\sum_{m=1}^M \frac{1}{z_m}\int_{E_m}(w_h(\bsx,\bsy)-W_h(\bsy))(v(\bsx)-V)\,{\rm d}S(\bsx).
\end{align*}
We can rewrite this as
$$
B((w_h(\cdot,\bsy),W_h(\cdot,\bsy)),(v,V))=L_h((v,V)),
$$
where $L_h\!:\mathcal H\to \mathbb R$, $L_h((v,V)):=-\int_D \alpha_h(\bsx,\bsy)\nabla u(\bsx,\bsy+h\boldsymbol e_j)\cdot \nabla v(\bsx)\,{\rm d}\bsx$, is a linear functional. Analogously to ~\cite[Theorem~4.2]{cohen} one can show that $L_h$ converges to $L_0$ in $\mathcal{H}$ as $h \rightarrow 0$. Therefore $(w_h, W_h)$ converges to $(w_0, W_0)$, which is the solution to
\begin{align*}
    \int_D a(\bsx, \bsy)\nabla w_0(\bsx, \bsy)\cdot \nabla v(\bsx) \,{\rm d}\bsx + \sum_{m=1}^M \frac{1}{z_m}\int_{E_m} (w_0(\cdot, \bsy)-W_0(\bsy))(v-V_m)\,{\rm d}S = L_0(v).
\end{align*}
Hence $\partial_{y_j}(u(\cdot, \bsy), U(\bsy)) = (w_0, W_0)$ exists in $\mathcal{H}$ and is the unique solution of the variational problem
\begin{align*}
    \int_D a(\bsx, \bsy)\nabla \partial_{y_j}u(\bsx, \bsy)\cdot \nabla v(\bsx) \,{\rm d}\bsx + \sum_{m=1}^M \frac{1}{z_m}\int_{E_m} (\partial_{y_j} u(\cdot, \bsy)-\partial_{y_j}U(\bsy))(v-V_m)\,{\rm d}S\\ = -\int_D \partial_{y_j}a(\bsx,\bsy)\nabla u(\bsx, \bsy) \cdot \nabla v(\bsx) \,{\rm d}\bsx
\end{align*}
for all $(v,V)\in\mathcal H$. This concludes the proof.
\end{proof}

By inductive reasoning, we obtain the existence of higher-order partial derivatives as a corollary.

\begin{corollary}\label{cor:corollary}
Let the assumptions of Lemma~\ref{lemma1} hold. Then $\partial_{\bsy}^{\bsnu}(u(\cdot,\bsy),U(\bsy))$ exists and belongs to $\mathcal H$ for all $\bsnu\in\mathscr F$ and $\bsy\in \Upsilon$.
\end{corollary}

Corollary~\ref{cor:corollary} enables us to obtain the higher-order partial derivatives of the solution to~\eqref{eq:variational} by formally differentiating the variational formulation on both sides with respect to the parametric variable. This yields the following recursive bound for the higher-order partial derivatives.

\begin{lemma}[Recursive bound]\label{lemma3}
Let the assumptions of Lemma~\ref{lemma1} hold. Let $\bsnu\in\mathscr F\setminus\{\mathbf 0\}$ and $\bsy\in \Upsilon$. Then
$$
\|\partial_{\bsy}^{\bsnu}(u(\cdot,\bsy),U(\cdot,\bsy))\|_{\mathcal H}\leq \frac{C_a}{\min\{a_{\min},\varsigma_+^{-1}\}}\!\sum_{\mathbf 0\neq \boldsymbol m\leq \bsnu}\!\binom{\bsnu}{\boldsymbol m}(|\boldsymbol m|!)^\sigma \boldsymbol\rho^{\boldsymbol m}\|\partial_\bsy^{\bsnu-\boldsymbol m}(u(\cdot,\bsy),U(\bsy))\|_{\mathcal H}.
$$
\end{lemma}
\begin{proof}
We differentiate the variational formulation~\eqref{eq:variational} on both sides with respect to $\bsy\in\Upsilon$ and use the Leibniz product rule, which yields
\begin{align*}
&\int_D a(\bsx,\bsy)\nabla \partial_{\bsy}^{\bsnu}u(\bsx,\bsy)\cdot \nabla v(\bsx)\,{\rm d}\bsx\\
&\quad +\sum_{m=1}^M \frac{1}{z_m}\int_{E_m} (\partial_\bsy^{\bsnu}u(\bsx,\bsy)-\partial_{\bsy}^{\bsnu}U(\bsx,\bsy))(v(\bsx)-V(\bsx))\,{\rm d}S(\bsx)\\
&=-\sum_{\mathbf 0\neq \boldsymbol m\leq \bsnu}\binom{\bsnu}{\boldsymbol m}\int_D \partial_{\bsy}^{\boldsymbol m}a(\bsx,\bsy)\nabla\partial_{\bsy}^{\bsnu-\boldsymbol m} u(\bsx,\bsy)\cdot \nabla v(\bsx)\,{\rm d}\bsx.
\end{align*}
Setting $(v,V)=(\partial_{\bsy}^{\bsnu}u(\cdot,\bsy),\partial_{\bsy}^{\bsnu}U(\bsy))\in\mathcal H$, we obtain
\begin{align*}
&\min\{a_{\min},\varsigma_+^{-1}\}\|(\partial_{\bsy}^{\bsnu}u(\cdot,\bsy),\partial_{\bsy}^{\bsnu}U(\cdot,\bsy))\|_{\mathcal H}^2\\
&\le \int_D a(\bsx,\bsy)|\nabla \partial_{\bsy}^{\bsnu}u(\bsx,\bsy)|^2\,{\rm d}\bsx+\sum_{m=1}^M \frac{1}{z_m}\int_{E_m} (\partial_\bsy^{\bsnu}u(\bsx,\bsy)-\partial_{\bsy}^{\bsnu}U(\bsy))^2\,{\rm d}S(\bsx)\\
&= -\sum_{\mathbf 0\neq \boldsymbol m\leq \bsnu}\binom{\bsnu}{\boldsymbol m}\int_D \partial_{\bsy}^{\boldsymbol m}a(\bsx,\bsy)\nabla\partial_{\bsy}^{\bsnu-\boldsymbol m} u(\bsx,\bsy)\cdot \nabla \partial_{\bsy}^{\bsnu}u(\bsx,\bsy)\,{\rm d}\bsx\\
&\leq \sum_{\mathbf 0\neq \boldsymbol m\leq \bsnu}\binom{\bsnu}{\boldsymbol m}\|\partial_{\bsy}^{\boldsymbol m}a(\cdot,\bsy)\|_{L^\infty(D)}\bigg(\int_D|\nabla\partial_{\bsy}^{\bsnu-\boldsymbol m} u(\bsx,\bsy)|^2\,{\rm d}\bsx\bigg)^{\frac12}\!\bigg(\int_D |\nabla \partial_{\bsy}^{\bsnu} u(\bsx,\bsy)|^2\,{\rm d}\bsx\bigg)^{\frac12} \\
&\leq \sum_{\mathbf 0\neq \boldsymbol m\leq \bsnu}\binom{\bsnu}{\boldsymbol m}C_a(|\boldsymbol m|!)^\sigma \boldsymbol\rho^{\boldsymbol m}\|\partial_\bsy^{\bsnu-\boldsymbol m}(u(\cdot,\bsy),U(\bsy))\|_{\mathcal H}\|\partial_\bsy^{\bsnu}(u(\cdot,\bsy),U(\bsy))\|_{\mathcal H},
\end{align*}
where we applied~\ref{a2}. The result follows by canceling the common factor on both sides and dividing both sides of the inequality by $\min\{a_{\min},\varsigma_+^{-1}\}$.
\end{proof}

In consequence, the recursive bound allows us to derive an {\em a priori} bound on the higher-order partial derivatives. To this end, we will make use of the following result for multivariable recurrence relations.
\begin{lemma}[cf.~{\cite[Lemma~3.1]{gk24plus}}]\label{gk}
Let $(\Xi_{\bsnu})_{\bsnu\in\mathscr F}$ and $\bsb=(b_j)_{j\geq 1}$ be sequences satisfying
$$
\Xi_{\mathbf 0}\leq c_0\quad\text{and}\quad  \Xi_{\bsnu}\leq c\sum_{\substack{\boldsymbol m\leq \bsnu\\ \boldsymbol m\neq\mathbf 0}}\binom{\bsnu}{\boldsymbol m}(|\boldsymbol m|!)^\sigma \boldsymbol b^{\boldsymbol m}\Xi_{\bsnu-\boldsymbol m}\quad\text{for all}~\bsnu\in\mathscr F\setminus\{\mathbf 0\},
$$
where~$c_0,c>0$  and $\sigma\geq1$. Then there holds
$$
\Xi_{\bsnu}\leq c_0c^{1-\delta_{|\bsnu|,0}}(c+1)^{\max\{|\bsnu|-1,0\}}(|\bsnu|!)^\sigma \boldsymbol b^{\bsnu}.
$$
\end{lemma}
Lemma~\ref{gk} immediately yields the following result when applied to the recurrence relation in Lemma~\ref{lemma3}.
\begin{lemma}[Inductive bound]\label{lemma:inductive}
Let the assumptions of Lemma~\ref{lemma1} hold. Let $\bsnu\in\mathscr F\setminus\{\mathbf 0\}$ and $\bsy\in \Upsilon$. Then
$$
\|\partial^{\bsnu}(u(\cdot,\bsy),U(\bsy))\|_{\mathcal H}\leq C_u(|\bsnu|!)^{\sigma}\boldsymbol b^{\bsnu},
$$
where $$C_u=\frac{C|I|}{\min\{a_{\min},\varsigma_+^{-1}\}}~\text{and}~ b_j:=\bigg(1+\frac{C_a}{\min\{a_{\min},\varsigma_+^{-1}\}}\bigg)\rho_j.$$
\end{lemma}

\subsection{Bayesian inverse problem}
In what follows, we truncate the parametric domain $\Upsilon$ to $(-1/2,1/2)^s$. We shall further assume that the coefficient function $a\!:D\times U\to\mathbb R$ satisfies $a(\cdot,\bsy)\in L^\infty(D)$  for all $\bsy\in(-1/2,1/2)^s$, with the additional regularity that $a(\bsx,\bsy)$ is well-defined for all $\bsx\in D$ and $\bsy\in U$. This ensures that $a(\cdot,\bsy)$ is not only an element of $L^\infty(D)$ but also admits pointwise evaluation. Bayesian inference can be used to express the solution to an inverse problem in terms of a high-dimensional posterior distribution. Specifically, we assume that we have some voltage measurements taken at the $M$ electrodes placed on the boundary of the computational domain $D$ corresponding to the {\em observation operator} $\mathcal O\!:(-1/2,1/2)^s\to \mathbb R^M$,
$$
\mathcal O(\bsy):=[U_1(\bsy),\ldots,U_M(\bsy)]^{\rm T}.
$$
Furthermore, we assume that the observations are contaminated with additive Gaussian noise, leading to the measurement model
$$
\boldsymbol\delta = \mathcal O(\bsy)+\boldsymbol\eta,
$$
where $\boldsymbol\delta\in\mathbb R^M$ are the noisy voltage measurements, $\mathcal O$ is the observation operator, $\bsy\in (-1/2,1/2)^s$ is the unknown parameter, and $\boldsymbol\eta\in\mathbb R^M$ is a realization of additive Gaussian noise.

We endow $\bsy$ with the uniform prior distribution $\mathcal U((-1/2,1/2)^s)$, assume that $\bseta\sim\mathcal N(0,\Gamma)$, where $\Gamma\in\mathbb R^{M\times M}$ is a symmetric, positive definite covariance matrix, and that $\bsy$ and $\bseta$ are independent. Then the posterior distribution can be expressed by applying Bayes' theorem and is given by
$$
\pi(\bsy\mid \boldsymbol\delta)=\frac{\pi(\boldsymbol\delta\mid \bsy)\pi_{\rm pr}(\bsy)}{Z(\boldsymbol\delta)},
$$
where $\pi_{\rm pr}(\bsy)=\chi_{(-1/2,1/2)^s}(\bsy)$, we define the unnormalized likelihood by
$$
\pi(\boldsymbol\delta \mid \bsy)={\rm e}^{-\frac 12 (\boldsymbol\delta-\mathcal O(\bsy))^{\rm T}\Gamma^{-1}(\boldsymbol\delta-\mathcal O(\bsy))},
$$
and $Z(\boldsymbol\delta):=\int_{(-1/2,1/2)^s}\pi(\boldsymbol\delta \,|\, \bsy)\,{\rm d}\bsy$. As our estimator of the unknown parameter $\bsy\in(-1/2,1/2)^s$, we consider the quantity of interest
\begin{align}\label{eq:ahatdef}
\widehat a=\frac{Z'(\boldsymbol x, \boldsymbol\delta)}{Z(\boldsymbol\delta)},\quad Z'(\boldsymbol x, \boldsymbol\delta):=\int_{(-1/2,1/2)^s} a(\bsx, \bsy) \,\pi(\boldsymbol\delta\mid \bsy)\,{\rm d}\bsy.
\end{align}

To obtain a parametric regularity bound for the normalization constant $Z(\boldsymbol\delta)$, there holds by Lemma~\ref{lemma:inductive}  that
$$
\|\partial_{\bsy}^{\bsnu}\mathcal O(\bsy)\|_{\mathbb R^M}\leq \|\partial^{\bsnu}(u(\cdot,\bsy),U(\bsy))\|_{\mathcal H}\leq C_{u}(|\bsnu|!)^\sigma\bsb^{\bsnu},
$$
with $C_u$ defined in Lemma~\ref{lemma:inductive}.

A parametric regularity bound for $Z(\boldsymbol\delta)$ can be obtained as a consequence of~\cite[Lemma~5.3]{ks24}.
\begin{lemma}\label{lemma:paramregnorm} Let the assumptions of Lemma~\ref{lemma1} hold. Let $\bsnu\in\mathbb N_0^s$ and $\bsy\in(-1/2,1/2)^s$. Then
$$
|\partial_{\bsy}^{\bsnu}\pi(\boldsymbol\delta\mid \bsy)|\leq \frac{1}{2^{\sigma}}\cdot 3.82^M (2^\sigma C_{u})^{|\bsnu|}\mu_{\min}^{-\frac{|\bsnu|}{2}}(|\bsnu|!)^\sigma\boldsymbol b^{\bsnu},
$$
where $0<\mu_{\min}\leq 1$ is a lower bound on the smallest eigenvalue of $\Gamma$.
\end{lemma}

\begin{proof}
We write
$$
\pi(\boldsymbol\delta\mid \bsy)=(f\circ h)(\bsy),
$$
where $f(\bsx)={\rm e}^{-\bsx^{\rm T}\bsx/2}$ and $h(\bsy)=\Gamma^{-1/2}(\boldsymbol \delta-\mathcal O(\bsy))$. By Cram\'er's inequality~\cite{abramowitz1964}
$$
\bigg|\frac{{\rm d}^k}{{\rm d}x^k}{\rm e}^{-x^2/2}\bigg|\leq 1.1 \sqrt{k!}\quad\text{for all}~x\in \mathbb R~\text{and}~k\in\mathbb N_0,
$$
there holds
\begin{align}\label{eq:cramer}
|\partial_{\bsx}^{\bsnu}{\rm e}^{-\bsx^{\rm T}\bsx/2}|\leq 1.1^M \sqrt{\bsnu!}\quad\text{for all}~x\in \mathbb R^M~\text{and}~k\in\mathbb N_0^M.
\end{align}
By the recursive Fa\`a di Bruno formula (cf.~\cite{savits}), we can write
\begin{align}\label{eq:fdb}
\partial_{\bsy}^{\bsnu}\pi(\boldsymbol\delta\mid \bsy)=\sum_{\substack{\boldsymbol\lambda\in\mathbb N_0^M\\ 1\leq |\boldsymbol\lambda|\leq |\bsnu|}}\partial_{\bsx}^{\boldsymbol\lambda}{\rm e}^{-\bsx^{\rm T}\bsx/2}\bigg|_{\bsx=\Gamma^{-1/2}(\bsy-\mathcal O(\bsy))}\kappa_{\bsnu,\boldsymbol\lambda}(\bsy),
\end{align}
where the sequence $(\kappa_{\bsnu,\boldsymbol\lambda}(\bsy))$ is defined recursively by setting 
\begin{align*}
&\kappa_{\bsnu,\mathbf 0}\equiv \delta_{\bsnu,\mathbf 0},\\
&\kappa_{\bsnu,\boldsymbol\lambda}\equiv 0\quad\text{if}~|\bsnu|<|\boldsymbol\lambda|~\text{or}~\boldsymbol\lambda\not\geq \mathbf 0~\text{(i.e., if $\boldsymbol\lambda$ contains negative entries),}\\
&\kappa_{\bsnu+\boldsymbol e_j,\boldsymbol\lambda}(\bsy)=\sum_{\ell\in{\rm supp}(\boldsymbol\lambda)}\sum_{\mathbf 0\leq\boldsymbol m\leq \bsnu}\binom{\bsnu}{\boldsymbol m}\partial_{\bsy}^{\boldsymbol m+\boldsymbol e_j}[h(\boldsymbol y)]_\ell \kappa_{\bsnu-\boldsymbol m,\boldsymbol\lambda-\boldsymbol e_\ell}(\bsy)\quad\text{otherwise}.
\end{align*}
Completely analogously to the proof of~\cite[Lemma~5.3]{ks24}, we obtain
\begin{align}\label{eq:kappa}
|\kappa_{\bsnu,\boldsymbol\lambda}(\bsy)|\leq \bigg(\frac{C_u}{\sqrt{\mu_{\min}}}\bigg)^{|\bsnu|}\bigg(\frac{|\bsnu|!(|\bsnu|-1)!}{\boldsymbol \lambda!(|\bsnu|-|\boldsymbol \lambda|)!(|\boldsymbol \lambda|-1)!}\bigg)^{\sigma}\bsb^{\bsnu}
\end{align}
for all $1\leq|\boldsymbol\lambda|\leq |\bsnu|$ and $\bsy\in(-1/2,1/2)^s$. Plugging the inequalities~\eqref{eq:cramer} and~\eqref{eq:kappa} into the expression~\eqref{eq:fdb} yields
\begin{align*}
|\partial_{\bsy}^{\bsnu}\pi(\boldsymbol\delta\mid \bsy)|&\leq 1.1^M\bigg(\frac{C_u}{\sqrt{\mu_{\min}}}\bigg)^{|\bsnu|}(|\bsnu|!(|\bsnu|-1)!)^{\sigma}\boldsymbol b^{\bsnu}\!\!\sum_{\substack{\boldsymbol\lambda\in\mathbb N_0^M\\ 1\leq |\boldsymbol\lambda|\leq |\bsnu|}}\!\bigg(\frac{1}{\sqrt{\boldsymbol\lambda!}(|\bsnu|-|\boldsymbol\lambda|)!(|\boldsymbol\lambda|-1)!}\bigg)^\sigma.
\end{align*}
It remains to estimate the multi-index sum:
\begin{align*}
&\sum_{\substack{1\leq |\boldsymbol\lambda|\leq |\bsnu|\\ \boldsymbol\lambda\in\mathbb N_0^M}}\bigg(\frac{1}{\sqrt{\boldsymbol\lambda!}(|\bsnu|-|\boldsymbol\lambda|)!(|\boldsymbol\lambda|-1)!}\bigg)^{\sigma} =\sum_{\ell=1}^{|\bsnu|}\bigg(\frac{1}{(|\bsnu|-\ell)!(\ell-1)!}\bigg)^{\sigma}\sum_{\substack{\boldsymbol\lambda\in\mathbb N_0^M\\ |\boldsymbol\lambda|=\ell}}\bigg(\frac{1}{\sqrt{\boldsymbol\lambda!}}\bigg)^{\sigma}\\
&\leq \bigg(\sum_{\ell=1}^{|\bsnu|}\bigg(\frac{1}{(|\bsnu|-\ell)!(\ell-1)!}\bigg)^\sigma\bigg) \bigg(\underset{=3.469506}{\underbrace{\sum_{\lambda=0}^\infty\frac{1}{\sqrt{\lambda!}}}}\bigg)^M\leq 3.47^M\bigg(\sum_{\ell=1}^{|\bsnu|}\frac{1}{(|\bsnu|-\ell)!(\ell-1)!}\bigg)^\sigma\\
&= 3.47^M\cdot \frac{2^{\sigma |\bsnu|-\sigma }}{((|\bsnu|-1)!)^\sigma }, 
\end{align*}
where we made use of $\sum_k a_k\leq \big(\sum_k a_k^{1/\sigma}\big)^{\sigma}$ for $a_k\geq0$ and $\sigma\geq1$ (cf.~\cite[Theorem~19]{hardy})  as well as the summation identity $\sum_{\ell=1}^\nu \frac{1}{(\nu-\ell)!(\ell-1)!}=\frac{2^{\nu-1}}{(\nu-1)!}$ (see~\cite[Lemma~A.1]{ks24}).
\end{proof}

For the parametric regularity of the term $Z'(\boldsymbol \delta)$, we obtain the following result.
\begin{lemma}\label{lemma:deriv} Let the assumptions of Lemma~\ref{lemma1} hold. Let $\bsnu\in\mathbb N_0^s$ and $\bsy\in(-1/2,1/2)^s$. Then
$$
|\partial_{\bsy}^{\bsnu} (a(\bsx, \bsy) \pi(\boldsymbol\delta\mid \bsy))|\leq \frac{C_a}{2^\sigma} \cdot 3.82^M((|\bsnu|+1)!)^\sigma\boldsymbol\beta^{\bsnu},
$$
where $0<\mu_{\min}\leq 1$ is a lower bound on the smallest eigenvalue of $\Gamma$ and we define the sequence $\boldsymbol\beta=(\beta_j)_{j\geq 1}$ by setting
\begin{align}\label{eq:betadef}
\beta_j:=\max\{1,2^\sigma C_{\rm init}\}\bigg(1+\frac{C_a}{\min\{a_{\min},\varsigma_{+}^{-1}\}}\bigg)\mu_{\min}^{-\frac{1}{2}}\rho_j.
\end{align}
\end{lemma}
\begin{proof}  Let $\bsnu\in\mathscr F\setminus\{\mathbf 0\}$. The Leibniz product rule yields
\begin{align*}
&\partial_{\bsy}^{\bsnu}\big(a(\bsx,\bsy){\rm e}^{-\frac12 (\boldsymbol\delta - \mathcal O(\bsy))^{\rm T}\Gamma^{-1}(\boldsymbol \delta - \mathcal O(\bsy))}\big) \\
=\,&\sum_{\boldsymbol m\leq \bsnu}\binom{\bsnu}{\boldsymbol m}\partial_{\bsy}^{\boldsymbol m}a(\bsx,\bsy)\partial_{\bsy}^{\bsnu-\boldsymbol m}{\rm e}^{-\frac12 (\boldsymbol \delta - \mathcal O(\bsy))^{\rm T}\Gamma^{-1}(\boldsymbol\delta-\mathcal O(\bsy))}\\
\leq\,& \frac{1}{2^\sigma} \cdot 3.82^MC_a\bigg(\max\{1,2^{\sigma}C_{\rm init}\}\bigg(1+\frac{C_a}{\min\{a_{\min},\varsigma_{+}^{-1}\}}\bigg)\mu_{\min}^{-\frac{1}{2}}\bigg)^{|\bsnu|}\boldsymbol\rho^{\bsnu}\\
&\quad\times\sum_{\boldsymbol m\leq \bsnu}\binom{\bsnu}{\boldsymbol m}(|\boldsymbol m|!)^{\sigma}((|\bsnu|-|\boldsymbol m|)!)^{\sigma}.
\end{align*}
The claim follows by noting that
\begin{align*}
&\sum_{\boldsymbol m\leq \bsnu}\binom{\bsnu}{\boldsymbol m}(|\boldsymbol m|!)^{\sigma}((|\bsnu|-|\boldsymbol m|)!)^{\sigma}=\sum_{\ell=0}^{|\bsnu|}(\ell!)^\sigma((|\bsnu|-\ell)!)^\sigma\sum_{\substack{\boldsymbol m\leq \bsnu\\ |\boldsymbol m|=\ell}}\binom{\bsnu}{\boldsymbol m}\\
&=\sum_{\ell=0}^{|\bsnu|}(\ell!)^\sigma((|\bsnu|-\ell)!)^\sigma\binom{|\bsnu|}{\ell}=|\bsnu|!\sum_{\ell=0}^{|\bsnu|}(\ell!)^{\sigma-1}((|\bsnu|-\ell)!)^{\sigma-1}\\
&\leq |\bsnu|!\sum_{\ell=0}^{|\bsnu|}(|\bsnu|!)^{\sigma-1}=(|\bsnu|!)^{\sigma}(|\bsnu|+1)\leq ((|\bsnu|+1)!)^{\sigma},
\end{align*}
where we used the fact that $\ell!(v-\ell)!\leq v!$ for natural numbers $v\geq \ell$ and that $x\mapsto x^{\sigma-1}$ is monotonically increasing for $\sigma\geq 1$ and $x>0$.
\end{proof}

\section{Quasi-Monte Carlo}\label{sec:qmc}
Let $F\!:[0,1]^s\to\mathbb R$ be a continuous function. We shall be interested in approximating integral quantities
$$
I_s(F):=\int_{[0,1]^s}F(\bsy)\,{\rm d}\bsy.
$$
The randomly shifted quasi-Monte Carlo QMC estimator of $I_s(F)$ is given by
$$
Q_{n, s}(F):=\frac{1}{nR}\sum_{r=1}^R \sum_{k=1}^n F(\{\bst_k+\boldsymbol\Delta^{(r)}\}),
$$
where $\boldsymbol\Delta^{(1)},\ldots,\boldsymbol\Delta^{(R)}$ are i.i.d.~random shifts drawn from $\mathcal U([0,1]^s)$, $\{\cdot\}$ denotes the componentwise fractional part, the cubature nodes are given by
$$
\bst_k:=\bigg\{\frac{k\boldsymbol z}{n}\bigg\},\quad k\in\{1,\ldots,n\},
$$
and $\boldsymbol z\in\{1,\ldots,n-1\}^s$ is called the {\em generating vector}.

Let $\boldsymbol\gamma=(\gamma_{\setu})_{\setu\subseteq\{1,\ldots,s\}}$ be a sequence of positive weights. We assume that the integrand $F$ belongs to a weighted Sobolev space with bounded first-order mixed partial derivatives, $\mathcal W_{s,\boldsymbol\gamma}$, the norm of which is given by
$$
\|F\|_{s,\bsgamma}:=\bigg(\sum_{\setu\subseteq\{1,\ldots,s\}}\frac{1}{\gamma_{\setu}}\int_{[0,1]^{|\setu|}}\bigg(\int_{[0,1]^{s-|\setu|}}\frac{\partial^{|\setu|}}{\partial\bsy_{\setu}}F(\bsy)\,{\rm d}\bsy_{-\setu}\bigg)^2\,{\rm d}\bsy_{\setu}\bigg)^{\frac12},
$$
where ${\rm d}\bsy_{\setu}:=\prod_{j\in\setu}{\rm d}y_j$ and ${\rm d}\bsy_{-\setu}:=\prod_{j\in\{1,\ldots,s\}\setminus\setu}{\rm d}y_j$.

The following well-known result shows that it is possible to construct generating vectors using a \emph{component-by-component (CBC)} algorithm~\cite{cbc1,actanumer,cbc2} satisfying rigorous error bounds.
\begin{theorem}[{cf.~\cite[Theorem~5.1]{kuonuyenssurvey}}]\label{thm:cbcerr}
Let $F$ belong to the weighted Sobolev space $\mathcal W_{s,\bsgamma}$ with weights $\bsgamma=(\gamma_{\setu})_{\setu\subseteq\{1,\ldots,s\}}$. A randomly shifted lattice rule with $n=2^m$ points, $m\geq 1$, in $s$ dimensions can be constructed by a CBC algorithm such that for $R$ independent random shifts and for all $\lambda\in(1/2,1]$, there holds
$$
\sqrt{\mathbb E_{\boldsymbol\Delta}|I_s(F)-Q_{n, s}(F)|^2}\leq \frac{1}{\sqrt R}\bigg(\frac{2}{n}\sum_{\varnothing\neq\setu\subseteq\{1,\ldots,s\}}\gamma_{\setu}^{\lambda}\bigg(\frac{2\zeta(2\lambda)}{(2\pi^2)^\lambda}\bigg)^{|\setu|}\bigg)^{\frac{1}{2\lambda}}\|F\|_{s,\bsgamma},
$$
where $\mathbb E_{\boldsymbol\Delta}$ denotes the expected value with respect to the uniformly distributed random shifts over $[0,1]^s$ and $\zeta(x):=\sum_{k=1}^\infty k^{-x}$ is the Riemann zeta function for $x>1$.
\end{theorem}

In what follows, we shall consider the QMC approximation of~\eqref{eq:ahatdef}. To this end, we let $\boldsymbol\Delta\in [0,1]^s$ denote a (random) shift and define
\begin{align*}
&Z_n'(\boldsymbol x,\boldsymbol \delta,\boldsymbol\Delta):=\frac{1}{n}\sum_{k=1}^n a(\bsx,\{\bst_k+\boldsymbol\Delta\}-\tfrac{\mathbf 1}{\mathbf 2})\pi(\boldsymbol\delta\mid \{\bst_k+\boldsymbol\Delta\}-\tfrac{\mathbf 1}{\mathbf 2}),\\
&Z_n(\boldsymbol \delta,\boldsymbol\Delta):=\frac{1}{n}\sum_{k=1}^n\pi(\boldsymbol\delta\mid \{\bst_k+\boldsymbol\Delta\}-\tfrac{\mathbf 1}{\mathbf 2}).
\end{align*}
We suppress the arguments and write $Z_n'=Z_n'(\boldsymbol x,\boldsymbol\delta,\boldsymbol\Delta)$ and $Z_n=Z_n(\boldsymbol\delta,\boldsymbol\Delta)$ whenever there is no risk of ambiguity.

Furthermore, we shall also be interested in the finite element approximations of $Z_n'$ and $Z_n$. If $D$ is a convex polyhedron, we can consider a family $\{V_h\}_{h}$ of finite element subspaces of $H^1(D)$, indexed by the mesh width $h>0$, which are spanned by continuous, piecewise linear finite element basis functions in such a way that each $V_h$ is obtained from an initial, regular triangulation of $D$ by recursive, uniform bisection of simplices. We denote by $Z_{n,h}'(\boldsymbol\delta,\boldsymbol\Delta)$ and $Z_{n,h}(\boldsymbol\delta,\boldsymbol\Delta)$ the corresponding quantities $Z_n'(\boldsymbol x,\boldsymbol\delta,\boldsymbol\Delta)$ and $Z_n(\boldsymbol\delta,\boldsymbol\Delta)$ when $(u(\cdot,\bsy),U(\bsy))\in\mathcal H$ is replaced by a finite element solution.

\section{Error analysis}\label{sec:error}
We proceed to prove a rigorous convergence rate for the QMC approximation of our quantity of interest. The following theorem showcases a suitable choice of {\em product and order dependent} (POD) weights.

\begin{theorem}\label{thm:qmcerror}
    Let the assumptions of Lemma~\ref{lemma1} hold. Suppose that $\boldsymbol \rho=(\rho_j)_{j\geq1}\in\ell^p(\mathbb N)$ for some $p\in(0,1)$. Then the root-mean-square error using a randomly shifted lattice rule with $n=2^m$ points, $m\geq 1$, obtained by a CBC algorithm with $R$ independent randoms shifts satisfies
        $$
        \sqrt{\mathbb E_{\boldsymbol\Delta}\bigg|\frac{Z'}{Z}-\frac{Z_n'}{Z_n}\bigg|^2}\leq C\cdot n^{-\min\{\frac{1}{p}-\frac12,1\}},
        $$
        where $C>0$ is independent of $s$ for the sequence of POD weights 
        \begin{align}\label{eq:podweights}
        &\gamma_{\mathfrak u}=\bigg(((|\mathfrak u|+1)!)^\sigma\prod_{j\in\mathfrak u}\frac{\beta_j}{\sqrt{2\zeta(2\lambda)/(2\pi^2)^\lambda}}\bigg)^{\frac{2}{1+\lambda}},\\
        &\lambda=\begin{cases}\frac{p}{2-p}&\text{if}~p\in(\frac23,\frac{1}{\sigma})\\
        \frac{1}{2-2\varepsilon}&\text{if}~p\in(0,\min\{\frac23,{\frac{1}{\sigma}}\}),~p\neq\frac{1}{\sigma},\end{cases} \notag
        \end{align}
        where $\varepsilon\in(0,1/2)$ is arbitrary and $\beta_j$ is defined by~\eqref{eq:betadef}.
\end{theorem}
\begin{proof} The ratio estimator satisfies the following bound:
\begin{align}
\bigg|\frac{Z'}{Z}-\frac{Z_n'}{Z_n}\bigg|&=\bigg|\frac{Z'Z_n-Z_n'Z}{ZZ_n}\bigg|=\bigg|\frac{Z'Z_n-Z'Z+Z'Z-Z_n'Z}{ZZ_n}\bigg|\notag\\
&\leq \frac{|Z'||Z-Z_n|}{|ZZ_n|}+\frac{|Z'-Z_n'|}{|Z_n|}\notag\\
&\leq a_{\max}C_{\boldsymbol\delta,\Gamma}^2|Z-Z_n|+C_{\boldsymbol\delta,\Gamma}|Z'-Z_n'|,\label{eq:ratioestim}
\end{align}
where $C_{\boldsymbol\delta,\Gamma}:=\big(\inf_{\bsy\in\Upsilon}{\rm e}^{-\frac12 (\boldsymbol\delta-\mathcal O(\bsy))^{\rm T}\Gamma^{-1}(\boldsymbol\delta - \mathcal O(\bsy))}\big)^{-1}$ is finite. In particular, we deduce that
$$
\mathbb E_{\boldsymbol\Delta}\bigg|\frac{Z'}{Z}-\frac{Z_n'}{Z_n}\bigg|^2\lesssim\mathbb E_{\boldsymbol\Delta}|Z-Z_n'|^2+\mathbb E_{\boldsymbol\Delta}|Z'-Z_n'|^2. 
$$

Since the derivative bound for $Z'$ dominates $Z$, we can focus on designing a QMC rule for $Z'$. To this end, plugging in the derivative bound from Lemma~\ref{lemma:deriv} into the CBC error bound in Theorem~\ref{thm:cbcerr} yields
\begin{align*}
\sqrt{\mathbb E_{\boldsymbol\Delta}|Z'-Z_n'|^2}\lesssim \bigg(\frac{2}{n}\bigg)^{\frac{1}{2\lambda}}C_{s,\boldsymbol\gamma,\lambda}^{\frac12},
\end{align*}
where
$$
C_{s,\boldsymbol\gamma,\lambda}=\bigg(\sum_{\varnothing\neq\mathfrak u\subseteq\{1,\ldots,s\}}\gamma_{\mathfrak u}^{\lambda}\bigg(\frac{2\zeta(2\lambda)}{(2\pi^2)^\lambda}\bigg)^{|\mathfrak u|}\bigg)^{\frac{1}{\lambda}}\sum_{\mathfrak u\subseteq\{1,\ldots,s\}}\frac{((|\setu|+1)!)^{2\sigma}\boldsymbol\beta_{\setu}^2}{\gamma_{\setu}}.
$$
The value of $C_{s,\boldsymbol\gamma,\lambda}$ is minimized by choosing (cf., e.g., \cite[Lemma~6.2]{kss12})
$$
\gamma_{\setu}=\bigg(((|\setu|+1)!)^\sigma\prod_{j\in\setu}\frac{\beta_j}{\sqrt{2\zeta(2\lambda)/(2\pi^2)^{\lambda}}}\bigg)^{\frac{2}{1+\lambda}},\quad \setu\subset\mathbb N,
$$
with the convention that $\gamma_{\varnothing}=1$. Plugging these weights into the expression for $C_{s,\boldsymbol\gamma,\lambda}$ yields
\begin{align*}
C_{s,\boldsymbol\gamma,\lambda}^{\frac{2\lambda}{\lambda+1}}&=\sum_{\setu\subseteq\{1,\ldots,s\}}((|\setu|+1)!)^{\sigma\lambda}\bigg(\frac{2\zeta(2\lambda)}{(2\pi^2)^\lambda}\bigg)^{\frac{1}{\lambda+1}}\prod_{j\in\setu}\beta_j^{\frac{2\lambda}{1+\lambda}}\\
&\leq \sum_{\ell=0}^\infty ((\ell+1)!)^{\sigma\lambda-1}\bigg(\frac{2\zeta(2\lambda)}{(2\pi^2)^\lambda}\bigg)^{\frac{1}{\lambda+1}}\bigg(\sum_{j=1}^\infty b_j^{\frac{2\lambda}{\lambda+1}}\bigg)^\ell=C(\boldsymbol\gamma,\lambda),
\end{align*}
where we used the inequality $\sum_{|\setu|=\ell,~\setu\subseteq\mathbb N}\prod_{j\in\setu}c_j\leq \frac{1}{\ell!}\big(\sum_{j=1}^\infty c_j\big)^\ell$ for $c_j\geq 0$.

It can be shown using the ratio test that the upper bound $C(\lambda,\boldsymbol\gamma)<\infty$ with the following choices:
\begin{itemize}
\item If $p\in(\frac23,\frac{1}{\sigma})$, we choose $\lambda=\frac{2}{2-p}$. This yields the \emph{dimension-independent} QMC convergence rate $O(n^{-\frac{1}{p}+\frac12})$.
\item If $p\in(0,\min\{\frac23,\frac{1}{\sigma}\})$, $p\neq \frac{1}{\sigma}$, we choose $\lambda=\frac{1}{2-2\varepsilon}$ for arbitrary $\varepsilon\in(0,\frac12)$. This yields the \emph{dimension-independent} QMC convergence rate $O(n^{-1+\varepsilon})$.
\end{itemize}
With the above choices for $\boldsymbol\gamma$ and $\lambda$, we  obtain the overall QMC convergence rate $O(n^{\max\{-\frac{1}{p}+\frac12,-1+\varepsilon\}})$.

\end{proof}

The dimension truncation error rate follows from the existing literature.
\begin{theorem}[{\cite[Theorem~4.3]{gk22}}]
Let the assumptions of Lemma~\ref{lemma1} hold. In addition, suppose that $\boldsymbol\rho=(\rho_j)_{j\geq 1}\in \ell^p(\mathbb N)$ for some $0<p<1$, $\rho_1\geq \rho_2\geq \cdots$, and define
$$
Z_\infty':=\lim_{s\to\infty}\int_{(-\frac12,\frac12)^s}a(\bsx,\bsy)\pi(\boldsymbol\delta\mid \bsy)\,{\rm d}\bsy,\quad Z_\infty:=\lim_{s\to\infty}\int_{(-\frac12,\frac12)^s}\pi(\boldsymbol\delta\mid \bsy)\,{\rm d}\bsy.
$$
Then
$$
\mathbb E\bigg[\frac{Z_\infty'}{Z_\infty}-\frac{Z'}{Z}\bigg]=O(s^{-\frac{2}{p}+1}),
$$
where the implied coefficient is independent of $s$.
\end{theorem}

The finite element error satisfies the following bound.
\begin{theorem}\label{thm:femerror}
Let the assumptions of Lemma~\ref{lemma1} hold. In addition, let $D$ be a convex, bounded polyhedron and suppose that $a(\cdot,\bsy)\in C^\infty(\overline{D})$ for all $\bsy\in (-1/2,1/2)^s$. Moreover, suppose that, for any $t\in[0,1/2)$, $u(\cdot,\bsy)\in H^{\frac32+t}(D)$ and $\|u(\cdot,\bsy)\|_{H^{\frac32+t}(D)}={O}(z_1^{-t-\varepsilon})$ as $z_1\to 0$ with $\varepsilon\in(0,1-t)$.  Then
$$
\sup_{\bsy\in (-\frac12,\frac12)^s}\bigg|\frac{Z_n'}{Z_n}-\frac{Z_{n,h}'}{Z_{n,h}}\bigg|\leq C|I|h^{1+2t},
$$
where $C>0$ is independent of $\bsy$.
\end{theorem}
\begin{proof} The ratio estimator can be bounded by a linear combination of $|Z_n'-Z_{n,h}'|$ and $|Z_n-Z_{n,h}|$ as in~\eqref{eq:ratioestim}. Below, we focus on the former term since the second term can be bounded in a completely analogous manner. Letting $\mathcal O_h(\bsy)$ denote the quantity $\mathcal O(\bsy)$ when $U$ is replaced with its finite element counterpart, we obtain
\begin{align*}
|Z_n'-Z_{n,h}'|&=\bigg|\int_{(-\frac12,\frac12)^s}a(\bsx,\bsy){[\rm e}^{-\frac12 \|\boldsymbol\delta-\mathcal O(\bsy)\|_{\Gamma^{-1}}^2}-{\rm e}^{-\frac12 \|\boldsymbol\delta-\mathcal O_h(\bsy)\|_{\Gamma^{-1}}^2}]\,{\rm d}\bsy\bigg|\\
&\leq a_{\max}\int_{(-\frac12,\frac12)^s}\left|{\rm e}^{-\frac12 \|\boldsymbol\delta-\mathcal O(\bsy)\|_{\Gamma^{-1}}^2}-{\rm e}^{-\frac12 \|\boldsymbol\delta-\mathcal O_h(\bsy)\|_{\Gamma^{-1}}^2}\right|\,{\rm d}\bsy\\
&\leq a_{\max}\sup_{\bsy\in (-\frac12,\frac12)^s}\big|\,\|\boldsymbol\delta-\mathcal O(\bsy)\|_{\Gamma^{-1}}-\|\boldsymbol\delta-\mathcal O_h(\bsy)\|_{\Gamma^{-1}}\,\big|\\
&\leq a_{\max}\sup_{\bsy\in (-\frac12,\frac12)^s}\|\mathcal O(\bsy)-\mathcal O_h(\bsy)\|_{\Gamma^{-1}}\\
&\leq a_{\max}\mu_{\min}^{-1/2}\sup_{\bsy\in (-\frac12,\frac12)^s}\|\mathcal O(\bsy)-\mathcal O_h(\bsy)\|.
\end{align*}
The final term can be bounded by a term of order $O(h^{1+2t})$ by following the argument given in~\cite[Section~5.2]{stratos}. The argument is based on the best approximation property
$$
\inf_{w\in \mathbb P_1}\|v-w\|_{H^1(D)}\leq Ch^{\frac12+t}\|v\|_{H^{\frac32+t}(D)},\quad t\in (0,\tfrac12),
$$
proved in \cite[Lemma~3.4]{stenberg} for bounded, polygonal domains $D\subset\mathbb R^d$, $d\in\{2,3\}$, as well as the inequality
$$
\|\mathcal O(\bsy)-\mathcal O_h(\bsy)\|\lesssim \inf_{w\in \mathbb P_1}\|u-w\|_{H^1(D)}.
$$
The claim follows as specified in~\cite{stratos} provided that the conditions $u(\cdot,\bsy)\in H^{\frac32+t}$ for any $t\in [0,1/2)$ and $\|u(\cdot,\bsy)\|_{H^{\frac32+t}(D)}=O(z_1^{-t-\varepsilon})$ hold.
\end{proof}

{\em Remark.} It is known that if $D$ has a $C^\infty$ boundary, the conditions $u(\cdot,\bsy)\in H^{\frac32+t}(D)$ for any $t\in [0,1/2)$ and $\|u(\cdot,\bsy)\|_{H^{\frac32+t}(D)}=O(z_1^{-t-\varepsilon})$ are satisfied (see~\cite{stratos}). Meanwhile, it is known that if $D$ is not a convex domain, then these conditions are violated (see~\cite{constabel}). However, in the case where $D$ is merely a convex and bounded polyhedron, there do not seem to be any results in the existing literature ensuring that $u(\cdot,\bsy)\in H^{\frac32+t}(D)$ for any $t\in[0,1/2)$ and $\|u(\cdot,\bsy)\|_{H^{\frac32+t}(D)}=O(z_1^{-t-\varepsilon})$, which is why we needed to make this an additional assumption in the statement of Theorem~\ref{thm:femerror}.

\medskip

We can merge the dimension truncation error, QMC error, and finite element error into a combined error bound.
\begin{theorem}
Let the assumptions of Theorems~\ref{thm:qmcerror}--\ref{thm:femerror} hold. Then we have the combined error estimate
\begin{align*}
    \sqrt{\mathbb E_{\boldsymbol\Delta}\bigg|\frac{Z_\infty'}{Z_\infty}-\frac{Z_{n,h}'}{Z_{n,h}}\bigg|^2}=O \left( s^{-\frac{2}{p}+1} + n^{-\min\{\frac{1}{p}-\frac12,1-\varepsilon\}} +  h^{1+2t} \right),
\end{align*}
for all~$\varepsilon\in(0,1/2)$.
\end{theorem}

\section{Numerical experiments}\label{sec:numex}
\begin{figure}[!t]\centering
\subfloat[Experiment 1 (ground truth)]{\label{fig:1a}\includegraphics[width=.45\textwidth,trim=2.1cm .8cm 2cm .9cm,clip]{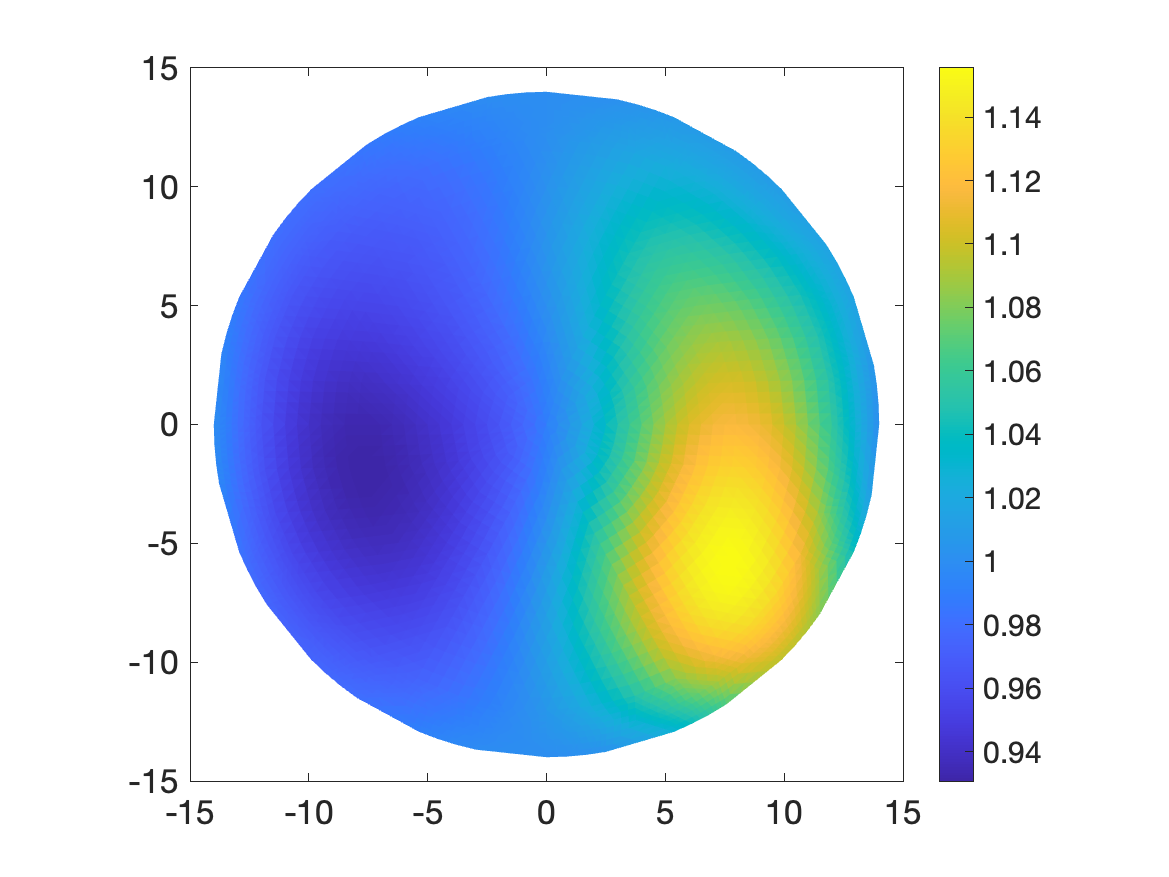}}
\subfloat[Experiment 2 (ground truth)]{\label{fig:1b}\includegraphics[width=.45\textwidth,trim=2.1cm .8cm 2cm .9cm,clip]{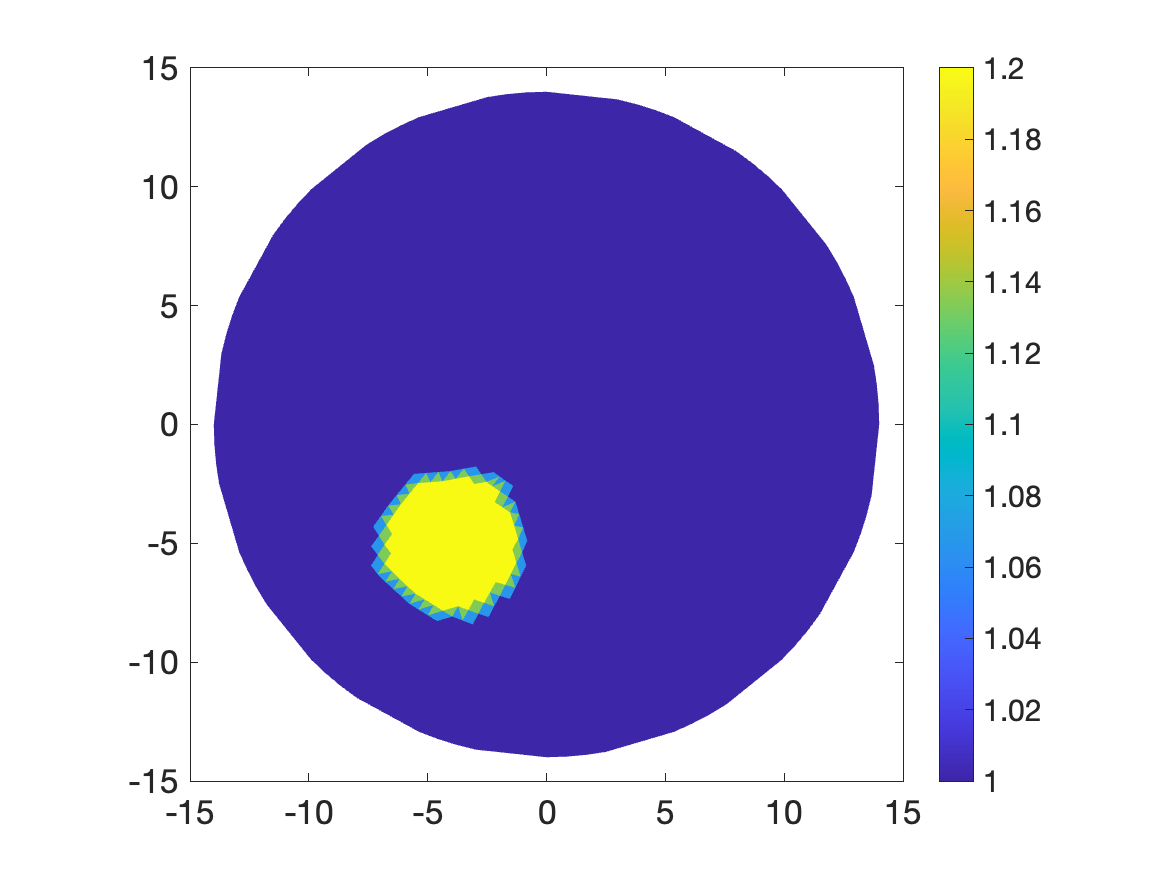}}
\caption{Two randomly generated ground truth conductivity fields which were used to simulate the noisy electrode measurements for numerical inversion.}\label{fig:groundtruth}
\subfloat[Experiment 1 (reconstruction)]{\label{fig:2a}\includegraphics[width=.45\textwidth,trim=2.1cm .8cm 2cm .9cm,clip]{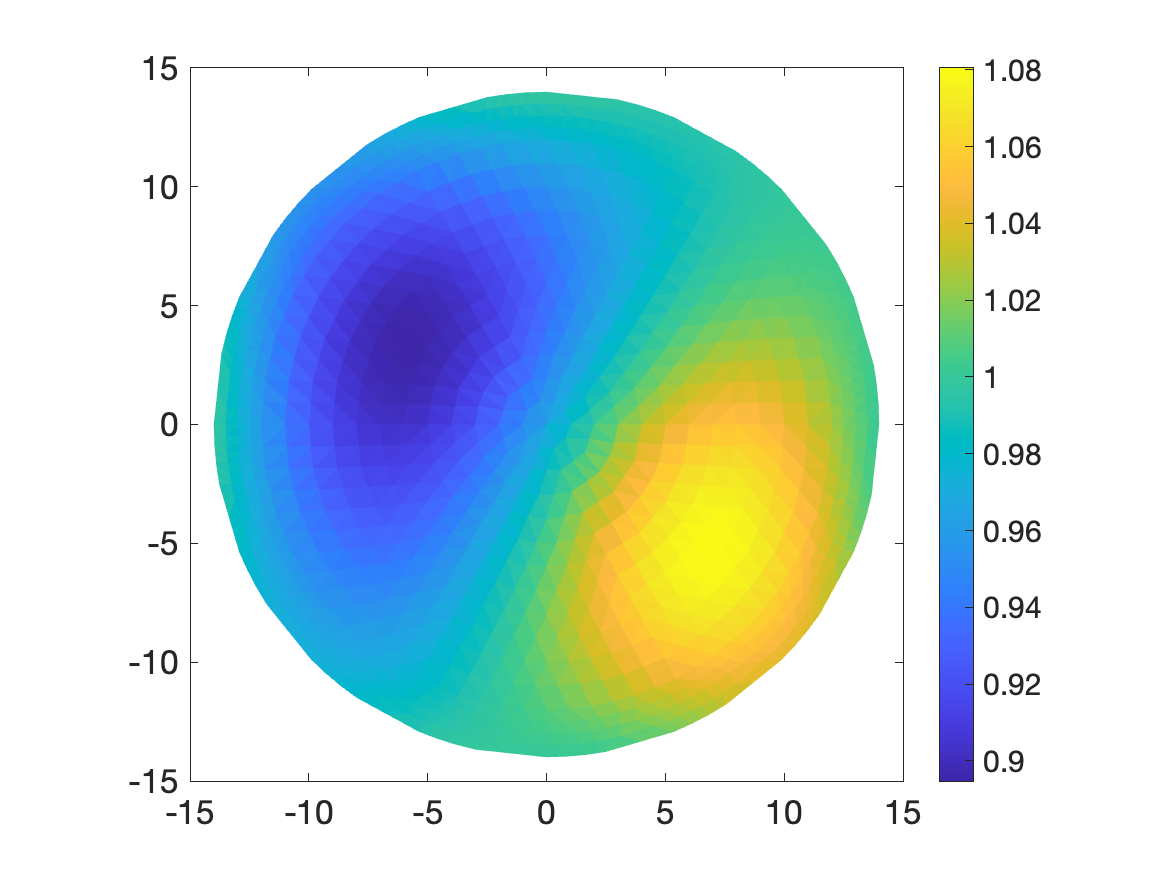}}
\subfloat[Experiment 2 (reconstruction)]{\label{fig:2b}\includegraphics[width=.45\textwidth,trim=2.1cm .8cm 2cm .9cm,clip]{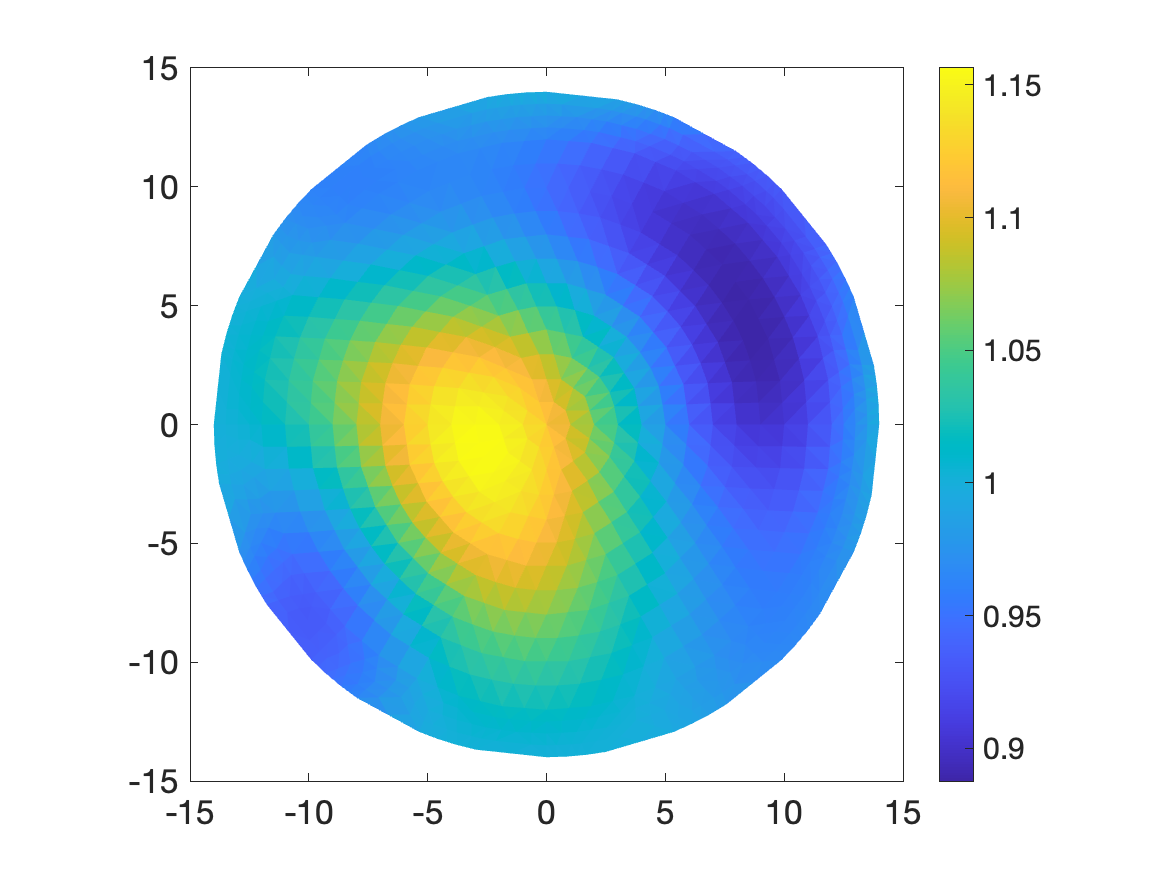}}
\caption{Reconstructions corresponding to noisy voltage measurements generated using the phantoms in Figure~\ref{fig:1a} and Figure~\ref{fig:1b}.
}
\end{figure}

\begin{figure}[!t]\centering
\subfloat[Experiment 1 (credible margin)]{\label{fig:creda}\includegraphics[width=.45\textwidth,trim=2.1cm .8cm 2cm .9cm,clip]{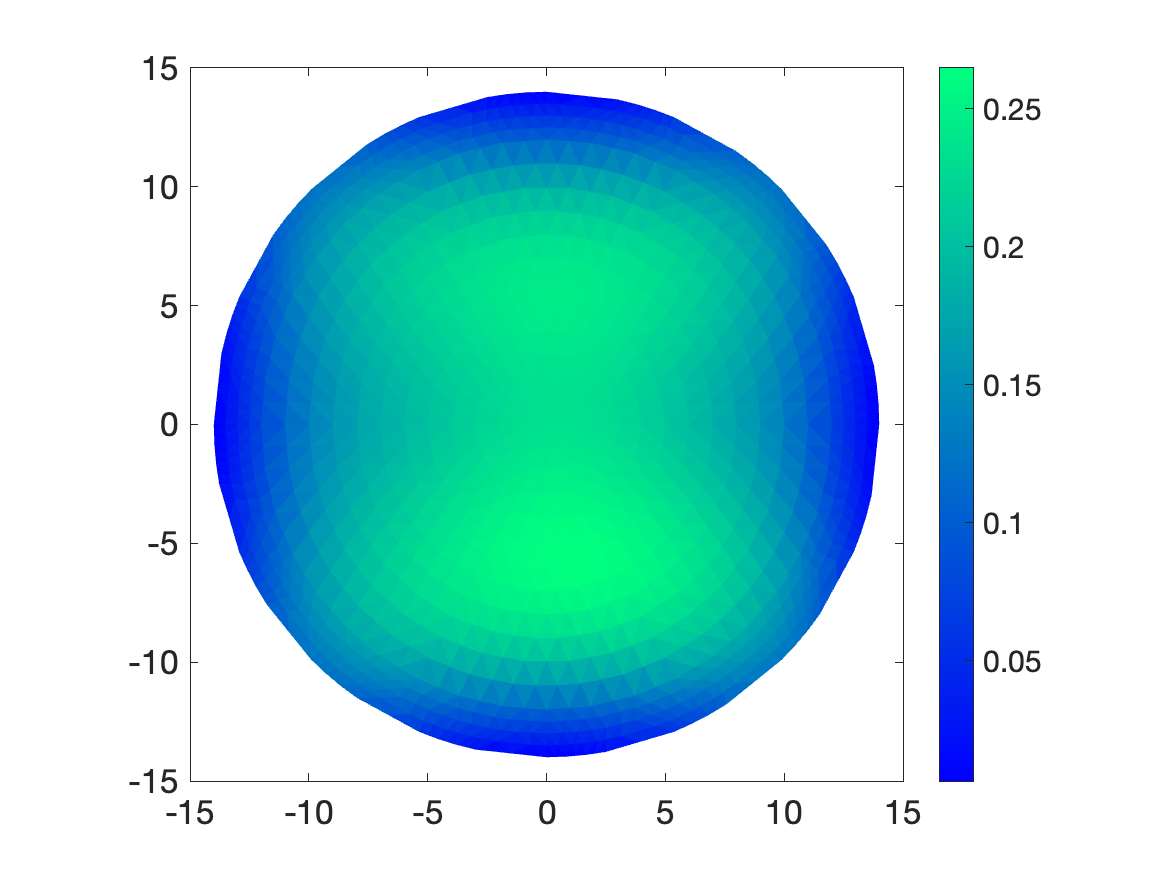}}
\subfloat[Experiment 2 (credible margin)]{\label{fig:credb}\includegraphics[width=.45\textwidth,trim=2.1cm .8cm 2cm .9cm,clip]{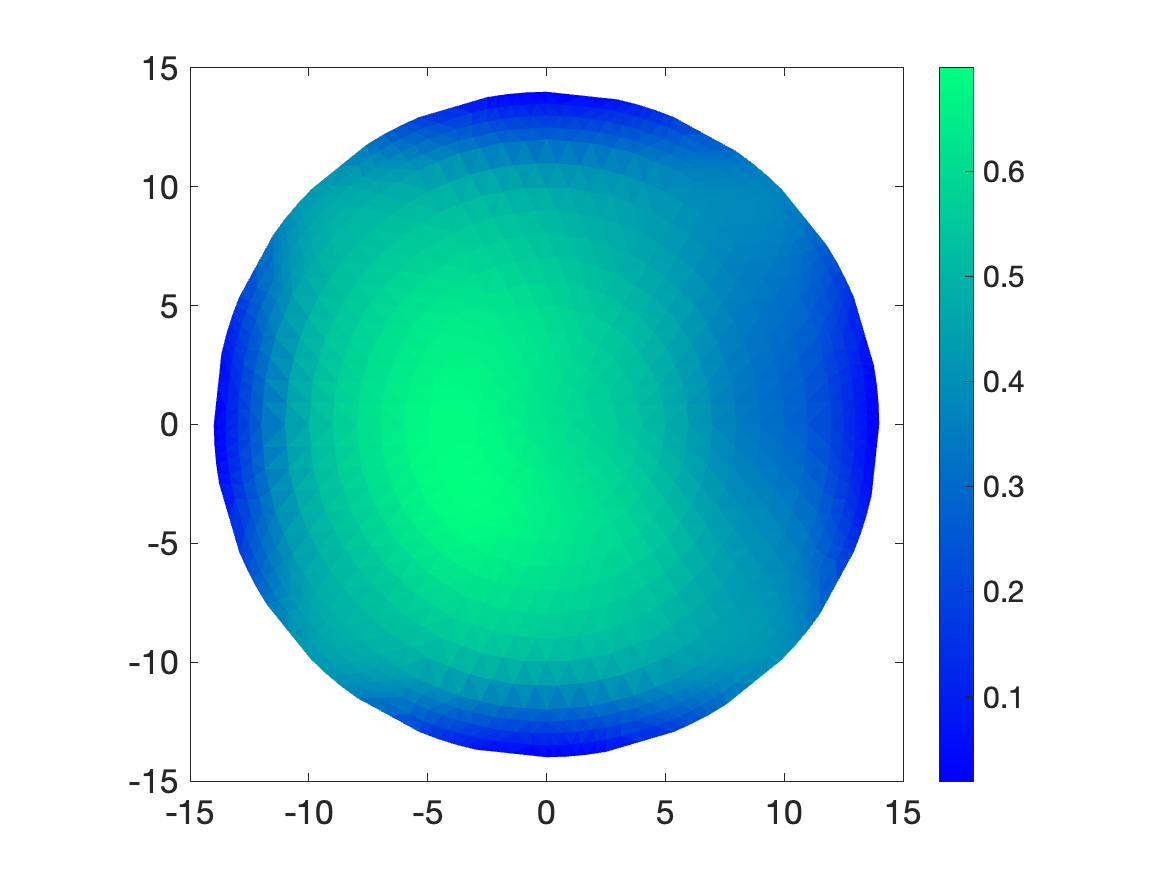}}
\caption{The credible margins of the approximate 95\% credible envelopes corresponding to the reconstructions in Figure~\ref{fig:2a} and Figure~\ref{fig:2b}.
}\label{fig:cred}
\end{figure}

We consider the problem of reconstructing the conductivity based on simulated measurement data consisting of voltages and currents on a set of electrodes placed on the boundary of the computational domain $D:=\{\bsx\in \mathbb R^2:\|\bsx\|\leq 14\}$. We let $\{E_k\}_{k=1}^L\subseteq \partial D$ be an array of $L:=16$ equispaced, non-overlapping electrodes of width $2.8$ on the boundary $\partial D$. We fix the current pattern $\boldsymbol I_k:=\boldsymbol e_1-\boldsymbol e_{k+1}\in\mathbb R_{\diamond}^L$, $k\in\{1,\ldots,L-1\}$. 

\begin{figure}[!t]
\subfloat[Experiment 1]{\label{fig:3a}\begin{tikzpicture}
    \begin{loglogaxis}[
        width=.45\textwidth,
        height=.45\textwidth,
        xlabel={number of nodes $n$},
        ylabel={R.M.S.~error},
        legend pos=north east,
        grid=both,
        grid style={dash pattern=on 1pt off 1pt on 1pt off 1pt},
        xmin=10000, xmax=1500000,
        ymin=3e-4, ymax=1e-1,
        xtick={10000, 100000, 1000000},
        ytick={1e-4, 1e-3, 1e-2, 1e-1}
    ]
        \addplot[
        color=blue,
        mark=*,
        only marks
    ] coordinates {
        (16384, 0.0347209)
        (32768, 0.0277225)
        (65536, 0.0190725)
        (131072, 0.0129361)
        (262144, 0.00755036)
        (524288, 0.00718684)
        (1048576, 0.00470984)
    };
    \addlegendentry{MC}
    \addplot[
        color=red,
        mark=square* ,
        only marks
    ] coordinates {
        (16384, 0.0148268)
        (32768, 0.00805214)
        (65536, 0.00480725)
        (131072, 0.00292949)
        (262144, 0.00148854)
        (524288, 0.000980165)
        (1048576, 0.000621175)
    };
   \addlegendentry{QMC}


        \addplot[
        color=red,
        domain=16384:1048576,
        samples=100,
        dashed,
        line width=1pt
    ] {24.212957 * (x^(-0.767824))};

    \addplot[
        color=blue,
        domain=16384:1048576,
        samples=100,
        dashed,
        line width=1pt
    ] {4.45029488 * (x^(-0.495653))};
    
    \end{loglogaxis}
\end{tikzpicture}}\subfloat[Experiment 2]{\label{fig:3a}\begin{tikzpicture}
    \begin{loglogaxis}[
        width=.45\textwidth,
        height=.45\textwidth,
        xlabel={number of nodes $n$},
        ylabel={R.M.S.~error},
        legend pos=north east,
        grid=both,
        grid style={dash pattern=on 1pt off 1pt on 1pt off 1pt},
        xmin=10000, xmax=1500000,
        ymin=1e-2, ymax=1e0,
        xtick={10000, 100000, 1000000},
        ytick={1e-4, 1e-3, 1e-2, 1e-1,1e0}
    ]
        \addplot[
        color=blue,
        mark=* ,
        only marks
    ] coordinates {
        (16384, 0.461215)
        (32768, 0.331697)
        (65536, 0.196404)
        (131072, 0.180036)
        (262144, 0.129976)
        (524288, 0.086197)
        (1048576, 0.0480702)
    };
    \addlegendentry{MC}
    \addplot[
        color=red,
        mark=square* ,
        only marks
    ] coordinates {
        (16384, 0.476305)
        (32768, 0.282797)
        (65536, 0.169824)
        (131072, 0.119845)
        (262144, 0.0825681)
        (524288, 0.0541753)
        (1048576, 0.0367871)
    };
   \addlegendentry{QMC}

    
        \addplot[
        color=red,
        domain=16384:1048576,
        samples=100,
        dashed,
        line width=1pt
    ] {151.65339276 * (x^(-0.603298))};

    \addplot[
        color=blue,
        domain=16384:1048576,
        samples=100,
        dashed,
        line width=1pt
    ] {65.6760282 * (x^(-0.509664))};
    
    \end{loglogaxis}
\end{tikzpicture}}
\caption{The R.M.S.~errors for reconstructions corresponding to the reconstructions illustrated in Figures~\ref{fig:2a}~and~\ref{fig:2b} with $R=16$ random shifts plotted alongside the corresponding least squares fits.}\label{fig:3}
\end{figure}

For the reconstruction of the target conductivity (``ground truth''), we use the parameterization
\begin{align}\label{eq:randoma1}
a(\bsx, \bsy) = \exp{\biggl(\sum_{j=1}^{20} y_j \psi_j(\bsx)\biggr)},\quad \bsx\in D,~\bsy\in (-1/2,1/2)^{20},
\end{align}
with
\begin{align}\label{eq:randoma2}
\psi_j(\bsx):=\frac{5}{(k_j^2+\ell_j^2)^{\vartheta}}\sin(\tfrac{1}{14}\pi k_jx_1)\sin(\tfrac{1}{14}\pi \ell_jx_2),\quad\vartheta>1,
\end{align}
where the sequence $(k_j,\ell_j)_{j\ge 1}$ is an ordering of the elements of $\mathbb N\times \mathbb N$ such that $\|\psi_1\|_{L^\infty(D)}\geq \|\psi_2\|_{L^\infty(D)}\geq \cdots$. %
We consider two different target conductivities. The electrode measurements for the first experiment correspond to a target conductivity that is a realization of the parametric model~\eqref{eq:randoma1}--\eqref{eq:randoma2} with $\vartheta=2$ using a randomly generated vector $\bsy\sim\mathcal U((-1/2,1/2)^{20})$. For the second experiment, the ground truth corresponds to the piecewise constant target conductivity $\bsx\mapsto 1+0.2\,\chi_{\|\bsx+[4,5]^{\rm T}\|\leq 3}(\bsx)$. These target conductivities are illustrated in Figure~\ref{fig:groundtruth}. In both experiments, the voltage measurements were simulated %
by solving the CEM forward model using a first order finite element discretization with maximum mesh diameter $h=0.748$. In addition, the measurements were contaminated with additive i.i.d.~Gaussian noise with covariance $\Gamma=0.014 I$. As the contact resistance of each electrode, we used the value $z_k=0.005$ for $k=1,\ldots,16$.

We approximate $\widehat a$ using the QMC ratio estimator with $n=2^{20}$ nodes. As the generating vector for the QMC rule, we used the ``off-the-shelf'' lattice generating vector~\cite[lattice-39101-1024-1048576.3600]{kuogeneratingvector}. For the computational inversion, we used a coarser finite element mesh with mesh width $h=1.496$. The computational inversion used the parametric model~\eqref{eq:randoma1}--\eqref{eq:randoma2} with $\vartheta=2$ in the first experiment and $\vartheta=1.3$ in the second experiment. The reconstructions for the two experiments are displayed in Figures~\ref{fig:2a} and~\ref{fig:2b}. In both cases, the profile and magnitude of the target conductivities are approximately recovered.%

An advantage of working in the Bayesian framework---in contrast to deterministic reconstruction algorithms such as the D-bar method~\cite{mueller2020d}---is that it is possible to estimate the uncertainty in $\widehat a$ via Bayesian credible envelopes. To this end, we have estimated the 95\% credible envelopes for the reconstructions. A conservative estimate for the $95\%$ credible envelope is $\mathbb E[\widehat a]\pm 4.47214\,\sqrt{{\rm Var}(\widehat a)}$ by Chebyshev's inequality. We computed the variance appearing in this expression using QMC pointwise in $\bsx\in D$ with $n=2^{20}$ lattice points. The obtained credible margins $4.47214\,\sqrt{{\rm Var}(\widehat a)}$ are displayed in Figure~\ref{fig:cred}.

In addition, we also assessed the convergence of the QMC estimator. To this end, we computed the value of $\widehat a$ for $n=2^\ell$, $\ell=10,\ldots,20$, and report the approximated root-mean-square (R.M.S.) errors
$$
\sqrt{\frac{1}{R(R-1)}\sum_{r'=1}^R\bigg\|\frac{1}{R}\sum_{r=1}^R \frac{Z_n'(\cdot,\boldsymbol\delta,\boldsymbol\Delta^{(r)})}{Z_n(\boldsymbol\delta,\boldsymbol\Delta^{(r)})}-\frac{Z_n'(\cdot,\boldsymbol\delta,\boldsymbol\Delta^{(r')})}{Z_n(\boldsymbol\delta,\boldsymbol\Delta^{(r')})}\bigg\|_{L^\infty(D)}^2},
$$
where we used $R=16$ random shifts $\boldsymbol\Delta^{(r)}\overset{\rm i.i.d.}{\sim}\mathcal U([0,1]^{20})$, $r=1,\ldots,16$.

We compare the performance of the QMC method to that of a Monte Carlo (MC) method. In the case of the MC method, the value of $\mathbb E[\widehat a]$ was obtained as a sample average over a random sample $\bsy\sim \mathcal U((-1/2,1/2)^s)$ of size $n=2^{\ell}$, $\ell=10,\ldots,20$. The corresponding R.M.S.~error was obtained by averaging these values over $16$ trials. The results are displayed in Figure~\ref{fig:3}. As expected, we observe the MC convergence rates $n^{-0.496}$ and $n^{-0.510}$ in Experiments 1 and 2, respectively, which are both extremely close to the theoretically expected MC converges rate $n^{-1/2}$. Meanwhile, the QMC method displays faster-than-MC cubature convergence rates in both experiments. In the case of the smooth target conductivity in Experiment 1, we observe an empirical convergence rate of $n^{-0.768}$ while in the case of the discontinuous conductivity in Experiment 2 we observe the empirical rate $n^{-0.603}$.

\section{Conclusions}\label{sec:conclusions}

Many studies on the application of QMC methods to PDE uncertainty quantification problems focus on simplified PDE models with homogeneous boundary conditions. In this paper, we considered a more realistic PDE model with intricate mixed boundary conditions. Specifically, we were able to obtain parametric regularity bounds for both the forward model and the Bayesian inverse problem. Our numerical results illustrate that QMC integration results in more accurate estimation of the state with less computational work than using a Monte Carlo method. The quality of the reconstructed features can potentially be improved using more sophisticated prior models, e.g., involving basis functions with local supports, coupling the reconstruction method with importance sampling or Laplace approximation as well as employing optimal experimental design. In addition, in real-life measurement configurations there are unknown quantities other than the conductivity that need to be recovered such as the contact resistances, electrode positions, and the shape of the computational domain. These could also be included as part of the inference problem.

\section*{Data availability statement}
No new data were created or analysed in this study.

\section*{Acknowledgements}
The work of Laura Bazahica and Lassi Roininen was supported by the Research Council of Finland (Flagship of Advanced Mathematics for Sensing, Imaging and Modelling grant 359183 and Centre of Excellence of Inverse Modelling and Imaging grant 353095).
\ifdefined\journalstyle
\section*{References}
\providecommand{\newblock}{}

\else
\bibliographystyle{unsrt}
\bibliography{qmc4eit}
\fi

\end{document}